\documentclass[12pt]{article}

\usepackage{amssymb}
\usepackage{amsmath, amsthm}
\usepackage{multirow}
\usepackage{hyperref}
\hypersetup{
  colorlinks   = true, 
  urlcolor     = blue, 
  linkcolor    = blue, 
  citecolor   = red 
}

\newtheorem{theorem}{Theorem}[section]
\newtheorem{proposition}[theorem]{Proposition}
\newtheorem{corollary}[theorem]{Corollary}
\newtheorem{lemma}[theorem]{Lemma}

\newcommand{\SL}{\mbox{SL}}
\newcommand{\gby}[1]{\left\langle#1\right\rangle}

\DeclareMathOperator{\tr}{tr}
\DeclareMathOperator{\wt}{wt}

\allowdisplaybreaks[2] 

\begin{document}

\title{One-point theta functions for vertex operator algebras}\author{Matthew Krauel\thanks{e-mail: krauel@csus.edu.  Research for this project was supported by the European Research Council (ERC) Grant agreement n. 335220 - AQSER.} \\ {\small Department of Mathematics and Statistics, California State University, Sacramento}}

\date{}
\maketitle

{\abstract
\noindent
One-point theta functions for modules of vertex operator algebras (VOAs) are defined and studied. These functions are a generalization of the character theta functions studied by Miyamoto and are deviations of the classical one-point functions for modules of a VOA. Transformation laws with respect to the group $\text{SL}_2(\mathbb{Z})$ are established.
}\\

\noindent
{\it Keywords:} vertex operator algebras; modular invariance

\section{Introduction\label{intro}}

 As the surge in interest of modular functions surrounding Monstrous Moonshine continued in the late 1980s and early 1990s, mathematical attention quickly expanded to a number theoretic study of the more general $n$-point functions associated to vertex operator algebras (VOAs). Zhu's celebrated (partial) solution \cite{Zhu} to the modularity of such functions consisted of developing recursion formulas enabling $n$-point functions to be written as a combination of classical elliptic and $(n-1)$-point functions, thus reducing the problem to the study of $1$-point functions. The latter of which, \textit{characters} (or \textit{graded dimensions}) associated to modules of VOAs are a special case. Motivated to express transformation properties of trace functions with automorphisms using only ordinary modules, Miyamoto \cite{Miy-Theta} studied a deviation of characters for the modules of a VOA and developed their transformation laws with respect to the group $\text{SL}_2 (\mathbb{Z})$. By exploiting these transformation laws, a number of works have been developed pertaining to elliptic genera \cite{DLiuMa}, Jacobi forms \cite{KrauelMason2, KrauelMason1}, and modular-invariance relations between shifted VOAs and orbifold theory \cite{Yamauchi-Primary}, among others. It is the aim of this paper to develop a $1$-point analogue of Miyamoto's theta functions and deduce their transformation laws with respect to $\text{SL}_2 (\mathbb{Z})$. This generalizes work of Miyamoto \cite{Miy-Theta} and Yamauchi \cite{Yamauchi-Primary}, and helps pave a way for studying $1$-point functions involving multiple variables.
 
  For a VOA $(V,Y(\ast),\textbf{1} ,\omega)$ of central charge $\mathbf{c}$, the vertex operator $Y(v,z):=\linebreak \sum_{n\in \mathbb{Z}} v(n) z^{-n-1}$ identifies infinitely many endomorphisms $v(n)$ to an element $v\in V$. The endomorphism $L(0)$ defined by setting $L(n):=\omega (n+1)$ supplies any ordinary $V$-module $M^\nu$ with an nonnegative integer grading $M^\nu =\bigoplus_{n \geq 0} M^\nu_{\lambda_\nu +n}$, where $\lambda_\nu$ is the conformal weight of $M^\nu$ and $M^\nu_{\lambda_\nu +n}=\{w\in M^\nu \mid L(0)w=(\lambda_\nu +n)w\}$. Here, and throughout the paper, we assume $V$ is rational and $C_2$-cofinite. Rationality implies $V$ has finitely many inequivalent irreducible modules, which we denote $M^1 ,\dots ,M^N$, and each of these possess such an $L(0)$-grading \cite[Theorem $8.1$]{DLM-Twisted}.
 
 For elements $J,K\in V_1$, Miyamoto \cite{Miy-Theta} introduced functions of the form
 \begin{equation}
 \Phi_\nu \left((J,K),\tau \right) := \tr_{M^\nu} e^{2\pi i \left(J(0) +\frac{\langle J,K\rangle}{2}\right)} q^{L(0)+K(0)+\frac{\langle K,K\rangle}{2} -\frac{\mathbf{c}}{24}},
 \end{equation}
 where $\langle \cdot ,\cdot \rangle$ is a bilinear form associated to $V$ (see \cite{Li-Bilinear} for more details about this form). Suppose (i) $V$ is a rational, $C_2$-cofinite VOA, (ii) $J,K \in V_1$ and satisfy $\alpha (n)\beta =\delta_{n,1} \langle \alpha ,\beta \rangle \textbf{1}$ for $\alpha ,\beta \in \{J,K\}$ and integers $n\geq 0$, (iii) $J(0)$ and $K(0)$ act semisimply on $V$ with eigenvalues in $\mathbb{C}$, (iv) $J$ and $K$ are quasi-primary, that is $L(1)J=L(1)K=0$, and (v) each function $\Phi_j$, $1\leq j\leq N$, converges for all $\tau \in \mathbb{H}$.
  Under these assumptions, Miyamoto proves that for each $V$-module $M^\nu$ and every $\gamma =\left(\begin{smallmatrix} a &b\\c &d \end{smallmatrix}\right) \in \text{SL}_2 (\mathbb{Z})$, there are scalars $A_{\nu ,j}^\gamma$ such that\footnote{See the Main Theorem on page $233$ of \cite{Miy-Theta}. The same theorem on page $223$ contains a typo.}
 \begin{equation}
 \Phi_\nu \left((J,K),\frac{a\tau +b}{c\tau +d} \right)= \sum_{j=1}^N A_{\nu ,j}^\gamma \Phi_j \left((bK+dJ,aK+cJ),\tau \right). \label{MiyTrans}
 \end{equation}
 This result generalizes and deviates from previous works. Most substantially, taking $(J,K)=(0,0)$ collapses to the level of characters, or the $n=1$ case with $a_1 =\textbf{1}$ considered by Zhu \cite[Theorem $5.3.2$]{Zhu}. In fact, the scalars $A_{\nu ,j}^\gamma$ are precisely the $S(\gamma, \nu,j)$ found in Zhu's theorem. We also mention the occurrence of this result with $K=0$ in the theory of vertex operator superalgebras \cite[Theorem $5.4$]{ADL-Modularity}.\\
 \indent To unravel the more general story of $1$-point functions and their deviations, we must introduce the \textit{zero mode} $o(v):=v(\wt v-1)$ of a homogeneous element $v\in V_{\wt v}$, where $\wt v$ denotes the weight of $v$ with respect to the $L(0)$-grading. We extend this definition linearly and note $o(v)$ is the unique endomorphism associated to $v$ which preserves the grading induced by $o(\omega)=L(0)$. Additionally, it is necessary to discuss the relationship between the original structure of a VOA $(V,Y(\ast),\textbf{1} ,\omega)$ and the change of coordinate VOA $(V,Y[\ast],\textbf{1} ,\widetilde{\omega})$ considered by Zhu \cite[Theorem $4.2.1$]{Zhu}. By setting $\widetilde{\omega}:=\omega -\frac{\mathbf{c}}{24}\textbf{1}$ and $Y[v,z]:=Y(e^{zo(\omega)}v, e^{z} -1) =:\sum_{n\in \mathbb{Z}} v[n] z^{-n-1}$ as in \cite{DLM-Orbifold}, $V$ acquires a different VOA structure with grading $V =\bigoplus_{n\in \mathbb{Z}} V_{[n]}$, where $V_{[n]}$ are now eigenspaces with respect to the $n=0$ case of the operators $L[n]=\widetilde{\omega}[n+1]$. If $v$ is homogeneous with respect to $L[0]$, we denote its weight by $\wt [v]$ for $v\in V_{[\wt [v]]}$. We refer the reader to Section \ref{SectionBasics} below and \cite{DLM-Orbifold, Zhu} for more details about these different VOA structures.\\
  \indent While Miyamoto's theorem continues to be used, a $1$-point analogue has thus far been lacking. Indeed, Miyamoto's results can be interpreted as the $v=\textbf{1}$ case of the functions
 \begin{equation}
 \Phi_j \left(v:(J,K),\tau \right) := \tr_{M^j} e^{2\pi i \left(o(J) +\frac{\langle J,K\rangle}{2}\right)} o(v) q^{o(\widetilde{\omega})+o(K)+\frac{\langle K,K\rangle}{2}}, \label{PHI}
 \end{equation}
 and the primary purpose of this paper is to develop transformation laws for functions of this type for any $v\in V$. To state our main result, we first define functions $\Phi_{j,\ell}$ for integers $\ell \geq 0$ and $v\in V$ by
 \begin{equation}
 \Phi_{j,\ell} \left(v:(J,K),\tau \right) := \frac{1}{\ell !}\Phi_j \left( (J+\tau K)[1]^\ell v :(J,K),\tau \right).
 \end{equation}
 
  Under precisely the same assumptions (i)--(v) above applied to the functions (\ref{PHI}), we obtain the following theorem.
 
\begin{theorem} \label{ThmMain}
Suppose (i) $V$ is a rational, $C_2$-cofinite VOA, (ii) $J,K \in V_1$ and satisfy $\alpha (n)\beta =\delta_{n,1} \langle \alpha ,\beta \rangle \textbf{1}$ for $\alpha ,\beta \in \{J,K\}$ and integers $n\geq 0$, (iii) $J(0)$ and $K(0)$ act semisimply on $V$ with eigenvalues in $\mathbb{C}$, (iv) $J$ and $K$ satisfy $L(1)J=L(1)K=0$, and (v) each function $\Phi_j$, $1\leq j\leq N$, converges for all $\tau \in \mathbb{H}$. Then for any $v\in V_{[\wt [v]]}$ and $\gamma =\left(\begin{smallmatrix} a &b\\c &d \end{smallmatrix}\right) \in \text{SL}_2 (\mathbb{Z})$, we have 
  \begin{equation}
  \begin{aligned}
 \Phi_j &\left(v:(J,K),\frac{a\tau +b}{c\tau +d} \right)\\
 &=(c\tau +d)^{\wt [v]} \sum_{k=1}^N A_{j,k}^\gamma \Phi_k \left(e^{cJ[1]+\frac{a\tau +b}{c\tau +d}cK[1]}v:(bK+dJ,aK+cJ),\tau \right),
 \end{aligned} \label{KTrans1}
 \end{equation}
  where the constants $A_{j,k}^\gamma$ are precisely those that arise in \cite[Theorem 5.3.2]{Zhu}. In other words,
 \begin{equation}
 \begin{aligned}
 \Phi_j &\left(v:(J,K),\frac{a\tau +b}{c\tau +d} \right) \\
 &= (c\tau +d)^{\wt [v]} \sum_{k=1}^N A_{j,k}^\gamma \sum_{\ell =0}^{\wt [v]} \Phi_{k,\ell} \left(v:(bK+dJ,aK+cJ),\tau \right)\left(\frac{c}{c\tau +d}\right)^\ell .
 \end{aligned} \label{KTrans2}
 \end{equation}
 Additionally, ignoring convergence and condition (v), these transformation rules hold so long as $J(0)v=K(0)v=0$.
\end{theorem}

 \indent Theorem \ref{ThmMain} establishes a type of quasi-modular form transformation property. Cleaner modular transformation laws similar to those satisfied by theta functions arise when $K[1]v=J[1]v=0$. When $(J,K)=(0,0)$, Theorem \ref{ThmMain} collapses to the main theorem in \cite{Zhu}, while the $v=\textbf{1}$ case gives the main result of \cite{Miy-Theta}. Meanwhile, the condition $(J,K)=(J,0)$ with $o(J)$ having rational eigenvalues produces a case of the modular transformations in \cite{DLM-Orbifold}. Other relevant results include Theorem 9.13 in \cite{HV-Susy}, which is a generalization of a special case of Theorem \ref{ThmMain}, as well as the similar, but independently developed Theorem 1.2 in \cite{AE-Relatively}.
 
  Additionally, we note how Theorem \ref{ThmMain} relates with another formulation of VOA theta functions due to Yamauchi \cite{Yamauchi-Primary}. By restricting attention to the case when $o(J)$ and $o(K)$ have rational eigenvalues, Yamauchi develops a generalization of Miyamoto's functions to incorporate automorphisms of $V$, as well as certain $1$-point elements. Utilizing theory related to shifted VOAs and orbifolds, he also establishes convergence on $\mathbb{H}$ for these functions by invoking work in \cite{DLM-Orbifold}. Specifically, for an element $v\in V_{[\wt [v]]}$, the $1$-point insertion $o\left(\sum_{s=0}^\infty p_s (J(1),J(2),\dots)v \right)$ is considered, where $p_s (J(1),J(2),\dots)$ are Schur polynomials defined by
 \begin{equation}
 z^{-J(0)}\Delta_J (z):=\exp \left(-\sum_{n= 1}^\infty \frac{J(n)}{n} (-z)^{-n} \right) =: \sum_{s=0}^\infty p_s (J(1),J(2),\dots)z^{s}. \label{IntroLi}
 \end{equation}
 Yamauchi employs results due to Li \cite{Li}, where $\Delta_J (z)$ is introduced and used to establish isomorphisms among twisted $V$-modules. It can be seen that
 \begin{equation}
 e^{J[1]}= z^{-J(0)}\Delta_J (1)=\sum_{s=0}^\infty p_s (J(1),J(2),\dots). \label{J[1]eq}
 \end{equation}
 Along with results such as Proposition 9 in \cite{MTZ}, this helps illuminate the connection between shifted VOA structures, orbifold theory, and the elliptic and modular form properties of $1$-point functions.\\
 \indent Set
 \begin{equation}
 \Psi_j \left(v:(J,K),\tau \right):= \Phi_j \left( e^{K[1]}v: (J,K),\tau \right). \label{YamFun}
 \end{equation}
 Using Theorem \ref{ThmMain} we obtain the following result analogous to the $(g,h)=(1,1)$ case of the main result of \cite{Yamauchi-Primary}, Theorem $1.1$, but allowing for complex eigenvalues of $o(J)$ and $o(K)$.
 
 \begin{corollary}\label{ThmMain2}
 Assume the same conditions as in Theorem \ref{ThmMain}. Then for any $v\in V_{[\wt [v]]}$ and $\gamma =\left(\begin{smallmatrix} a &b\\c &d \end{smallmatrix}\right) \in \text{SL}_2 (\mathbb{Z})$, we have
  \begin{equation}
 \Psi_\nu \left(v:(J,K),\frac{a\tau +b}{c\tau +d} \right)
 = (c\tau +d)^{\wt [v]} \sum_{k=1}^N A_{j,k}^\gamma \Psi_{k} \left(v:(bK+dJ,aK+cJ),\tau \right),
 \label{KTrans3}
 \end{equation}
 where the constants $A_{j,k}^\gamma$ are precisely those that arise in \cite[Theorem 5.3.2]{Zhu}.
 \end{corollary} 
 
 Both Theorem \ref{ThmMain} and Corollary \ref{ThmMain2} are proved in Section \ref{SectionProofs}. We recall they assume convergence. However, as noted before, some situations of convergence are known. For example, a statement of convergence for functions with $v=\textbf{1}$ is made in Proposition $1.8$ of \cite{DLiuMa}. Therefore, $\Phi_j$ of any $v$ corresponding to the derivative with respect to $\tau$ of $\Phi_j (\textbf{1}:(J,K),\tau)$ also converge on this domain. Meanwhile, as mentioned above, the convergence of the functions $\Psi_j$ for all $v$ on $\mathbb{H}$ when $o(J)$ and $o(K)$ have rational eigenvalues is in \cite{Yamauchi-Primary}. 
 
 In Section \ref{SectionExamples} we explore an example and, using Theorem \ref{ThmMain}, establish another proof of the quasi-modular properties for a partial derivative of the Jacobi theta functions $\vartheta_j$. Additionally, just as in the case of Miyamoto's original work, the case $K=0$ is of particular importance, and the transformation laws of Theorem \ref{ThmMain} can be exploited to establish a portion of the transformation laws for (quasi) Jacobi forms for strongly regular VOAs. This will be explored elsewhere.


\section{Preliminaries and notation\label{SectionBasics}}


\subsection{Elliptic functions}
 
 Let $q_x$ denote $e^{2\pi ix}$ for a variable $x$.
 Define the functions $P_k (\tau ,z)$ for $k\geq 1$ by
 \begin{equation}
 P_k (\tau ,z):= \frac{1}{(k-1)!} \sum_{n\in \mathbb{Z}\setminus \{0\}} \frac{n^{k-1} {q}_z^{n}}{1- q^n}. \notag 
 \end{equation}
 These functions, when multiplied by $(2\pi i)^{k}$, are the functions $P_k (q_z ,q)$ in \cite{Miy-Theta, Zhu}.
 For $\gamma = \left( \begin{smallmatrix} a&b \\c&d \end{smallmatrix} \right) \in \SL_2 (\mathbb{Z})$ and $k\geq 3$, they have the transformation properties
 \[
 P_k \left(\frac{a\tau +b}{c\tau +d} ,\frac{z}{c\tau +d} \right) =(c\tau +d)^k P_k (\tau ,z).
 \]
 Meanwhile, for $k=1,2$, we have
 \[
 P_1 \left(\frac{a\tau +b}{c\tau +d} ,\frac{z}{c\tau +d} \right) =(c\tau +d) P_1 (\tau ,z) +\frac{c\tau +d}{2} - cz -\frac{1}{2}
 \]
 and
 \[
 P_2 \left(\frac{a\tau +b}{c\tau +d} ,\frac{z}{c\tau +d} \right) =(c\tau +d)^2 P_2 (\tau ,z) - \frac{c(c\tau +d)}{2\pi i}.
 \] 


\subsection{Weight one elements of a VOA}

The relationship between modes of an element $u\in V$ under the original VOA structure and that of the change of coordinate VOA is given by
\begin{equation}
 u[m] = m! \sum_{i\geq m} c(\wt u,i,m)u(i), \label{bracket1}
\end{equation}
where $c(\wt u,i,m)$ are defined by the coefficients of the series 
\begin{equation}
 m!\sum_{i\geq m} c(\wt u,i,m)x^i := (\ln (1+x))^m (1+x)^{\wt u -1} \label{bracket2}
\end{equation}
expanded in the variable $x$ (see, for example, \cite[Lemma $4.3.1$]{Zhu}). Since $\ln (1+x) =-\sum_{n=1}^\infty \frac{(-1)^n}{n}x^n$, we have $c(1,i,1)=-\frac{(-1)^i}{i}$ and combining with (\ref{bracket1}) we find (\ref{J[1]eq}) holds as stated in the introduction. Additionally, (\ref{bracket1}) implies $L[0]=L(0)+\sum_{i\geq 0} k_i L(i)$ for some scalars $k_i$. It follows that if $u\in V_1$ and $L(1)u =0$, then $u\in V_{[1]}$ as well. We also note that for $u\in V_1$, (\ref{bracket2}) gives $c(1,i,0)=\delta_{i,0}$ and thus $u[0]=u(0)$. Suppose $u,v\in V_1$ and $u(k)v=\delta_{k,1} \langle u,v\rangle \textbf{1}\in V_0$ for $k\geq 0$. Since (\ref{bracket2}) gives $c(1,1,1)=1$ under this assumption, (\ref{bracket1}) implies $u[k]v=\delta_{k,1}u(1)v=\delta_{k,1} \langle u,v\rangle \textbf{1}$. Additionally, we have for all $s,t \geq 0$ that
 \begin{equation}
 u [s]v [t] = v [t] u [s] +[u [s] , v [t]] = v [t] u [s] +\sum_{k\geq 0} \binom{s}{k} (u [k] v)[s+t-k] 
 = v [t] u [s].
 \label{commutation}
 \end{equation} 
 
 As these simple results will be referenced later, we collect them in the following lemma.
 
 \begin{lemma}\label{LemWt1}
  Suppose $v_1 ,\dots ,v_n \in V_1$ and for every $1\leq i,j \leq n$ we have
 \begin{enumerate}
 \item[(a)] $v_i (k) v_j = \delta_{k,1} \langle v_i ,v_j \rangle \textbf{1}$ for $k\geq 0$, and \label{C3}
 \item[(b)] $v_1 ,\dots ,v_n$ are quasi-primary (and thus primary), that is $L(1)v_j=0$.
 \end{enumerate}
 Then for any $1\leq i,j \leq n$ and $v\in V$,
 \begin{enumerate}
 \item $v_j (0)v=v_j [0]v$,
 \item $v_j \in V_{[1]}$,
 \item $v_i [k] v_j =\delta_{k,1} \langle v_i,v_j \rangle \textbf{1}$ for $k\geq 0$, and
 \item $v_i [s]v_j [t] =v_j [t]v_i[s]$ for all $s,t \geq 0$.
 \end{enumerate}
 \end{lemma} 
 
 

\section{Proofs of theorems\label{SectionProofs}}

The proof of Theorem \ref{ThmMain} follows the ideas developed by Miyamoto in \cite{Miy-Theta}. In particular, we are ultimately interested in the transformation properties for expressions of the form
 \[
 \sum_{\ell_1 ,\ell_2 =0}^\infty \frac{1}{\ell_1! \ell_2!} \tr_{M^\nu} o(J)^{\ell_1} o(\tau K)^{\ell_2} o(v) q^{o(\widetilde{\omega})}
 \]
 under the action of $\text{SL}_2 (\mathbb{Z})$ (see Subsection \ref{Step5} below), where $M^{\nu}$ is a $V$-module from the list $\{M^1 ,\dots ,M^N \}$ discussed in the introduction. Before undertaking this, however, we accumulate some necessary results in Subsections \ref{Step1}--\ref{Step4}. Moreover, before considering $o(J)^{\ell_1}$ and $o(K)^{\ell_2}$, we instead consider arbitrary elements satisfying the the assumptions of Lemma \ref{LemWt1}.\\
 \indent For the entirety of Section \ref{SectionProofs}, we assume $v_1 ,\dots ,v_n$ satisfy the assumptions of Lemma \ref{LemWt1} and $\psi$ is an arbitrary grade-preserving endomorphism on the underlying vector space. Set 
 \begin{equation}
 \begin{aligned}
 &S^{\nu} (\psi ;z_1 ,\dots ,z_n ,\left\{ v,x\right\},\tau) \\
 &\quad :=\tr_{M^{\nu}} \psi Y \left( {q}_{z_1}^{o(\omega)} v_1 , {q}_{z_1} \right) \cdots Y \left({q}_{z_n}^{o(\omega)} v_n , {q}_{z_n} \right)Y \left({q}_{x}^{o(\omega)} v , {q}_{x} \right) q^{o(\widetilde{\omega})},
 \end{aligned} \label{npt}
 \end{equation}
 where again $q_w =e^{2\pi iw}$ for a variable $w$. Many of the upcoming subsections begin with new notation which will carry over to subsequent subsections. 
 
 
\subsection{Step one\label{Step1}}

 The first step is to express the $(n+1)$-point functions (\ref{npt}) as linear combinations of functions $P_k (\tau ,z)$ and VOA trace functions of $n$ zero modes and one vertex operator.\\
 \indent For an element $v_i \in \{v_1 ,\dots ,v_n\}$, set $\phi (v_i):=\phi (v_i,x-z_i,\tau) =\sum_{m_i \geq 1} P_{m_i}(x-z_i,\tau) v_i [m_i -1]$. Since $v_i [s]v_j [t] =v_j [t]v_i[s]$, we observe that $\phi (v_i)\phi (v_j) =\phi (v_j) \phi (v_i)$  for all $1\leq i,j \leq n$. For a set $U$, let $I(U)$ denote the set of all elements $\sigma\in \text{Sym}(U)$ such that $\sigma^2 =1$. Here $\text{Sym}(U)$ denotes the symmetric group with identity $1$ of the set $U$. In the case $U=\{ 1,\dots ,n\}$, we often write $I(n)$ in place of $I(\{1,\dots ,n\})$. For $\sigma \in \text{Sym}(U)$ set
 \begin{align*}
 m(\sigma) &:= \{i\in U \mid \sigma (i)\not =i \} \quad \text{and} \\
 f(\sigma) &:= \{i\in U \mid \sigma (i) =i \}.
 \end{align*}
 Finally, set $x_i :=x-z_i$ and $z_{i,j}:=z_i -z_j$.

\begin{proposition}\label{Prop1}
Suppose $v_1 ,\dots ,v_n$ satisfy the assumptions of Lemma \ref{LemWt1}, $v\in V$, and $[\psi ,v_j(m)]=0$ for each $1\leq j\leq n$ and $m\in \mathbb{Z}$. For $n\geq 1$ we have
 \begin{align}
 &S^{\nu} \left(\psi ;z_1 ,\dots ,z_n ,\left\{ v,x\right\},\tau \right)\notag \\
 &=\sum_{\sigma \in I(n)} \prod_{j<\sigma (j)} \langle v_j ,v_{\sigma (j)} \rangle P_2 \left( z_{\sigma (j),j} ,\tau \right) \sum_{U\subseteq f(\sigma)} S^{\nu} \left( \prod_{r\in U} o(v_r); \left\{ \left(\prod_{s\in f(\sigma)\setminus U} \phi (v_s)\right) v,x\right\} ,\tau \right). \notag
 \end{align}
\end{proposition}

\begin{proof} 
 Throughout this proof we suppress the notation relying on the module $M^\nu$. That is, we write $S$ and $\tr$ instead of $S^{\nu}$ and $\tr_{M^{\nu}}$, respectively. Without loss of generality we may assume $v\in V_{\wt v}$. For $k\in \mathbb{Z}$, a similar calculation as in the proof of Proposition $4.1$ in \cite{Miy-Theta} gives
 \begin{align*}
 S (\psi v_1 (k) {q}_{z_1}^{-k} ; z_2 ,\dots ,z_n ,\left\{ v,x\right\}, \tau) 
 &=k{q}_{z_{j,1}}^{k} \sum_{j=2}^n \gby{v_1 ,v_j} S (\psi ;z_2 ,\dots ,\widehat{z_j}, \dots ,z_n, \left\{ v,x\right\},\tau) \\
&\hspace{5mm}+  {q}_{ x_1}^{k} \sum_{i \geq 0} \binom{k}{i} S  (\psi  ;z_2 ,\dots ,z_n, \{v_1 (i)v,x\} ,\tau) \notag\\
 &\hspace{5mm}+ q^k S  (\psi v_1 (k) {q}_{z_1}^{- k} ;z_2 ,\dots ,z_n, \left\{ v,x\right\},\tau)  ,
 \end{align*} 
 where $\widehat{X}$ denotes the omission of the term $X$. Then for $k\not =0$, using
 \[
 \sum_{i\geq 0} \binom{k}{i} v_1 (i) =\sum_{m \geq 0} \frac{(k +1 -\wt v_1)^m}{m!} v_1 [m],
 \]
 (which can be deduced from \eqref{bracket1} and \eqref{bracket2}) we have
 \begin{align*}
 S &(\psi v_1 (k) {q}_{z_1}^{- k} ;z_2 ,\dots ,z_n ,\left\{ v,x\right\},\tau)\\
 &= \sum_{j=2}^n \gby{v_j ,v_1} \frac{k {q}_{z_{j,1}}^{k}}{1- q^k} S (\psi ;z_2 ,\dots ,
 \widehat{z_{j}}, \dots ,z_n ,\left\{ v,x\right\},\tau) \notag \\
 &\hspace{5mm} + \sum_{m \geq 0} \frac{k^m}{m!} \frac{{q}_{ x_1}^{k}}{1- q^k}  S(\psi ;z_2 ,\dots ,z_n ,\{v_1 [m]v,x\}, \tau).
 \end{align*}
 Therefore, we find
 \begin{align*}
 &S (\psi ;z_1 ,z_2 ,\dots ,z_n ,\left\{ v,x\right\},\tau) \\
 &=S  (\psi v_1 (0); z_2 ,\dots ,z_n ,\left\{ v,x\right\},\tau) +\sum_{k\in \mathbb{Z}\setminus \{0\}}
 S (\psi v_1 (k) {q}_{z_1}^{- k} ;z_2 ,\dots ,z_n ,\left\{ v,x\right\},\tau) \\
 &= S  (\psi v_1 (0);z_2 ,\dots ,z_n ,\left\{ v,x\right\},\tau)+ \sum_{j=2}^n \gby{v_1 ,v_j} P_2 (z_{j,1}, \tau)
 S (\psi ;z_2 ,\dots ,\widehat{z_{j}},\dots ,z_n ,\left\{ v,x\right\},\tau) \notag \\
 &\hspace{5mm} + \sum_{m \geq 1} P_{m} (x_1 ,\tau) S (\psi ;z_2 ,\dots ,z_n ,\left\{v_1 [m-1]v,x\right\}, \tau).
 \end{align*}
 Repeating the steps gives the desired result.
 \end{proof}
 
Note that rearranging the order of the vertex operators in the previous theorem leads to a different, but similar result (see Lemma 8.5 in \cite{DLM-Orbifold}, for example).


\subsection{Step two\label{Step2}}
 
 The next step is to incorporate the action of $\text{SL}_2 (\mathbb{Z})$ into the the terms of Proposition \ref{Prop1}. We do this by utilizing Zhu's modularity theorem for $n$-point functions \cite{Zhu}. More precisely, we use the following $g=h=1$ case of Assertion $2$ in the proof of Theorem $4.10$ in \cite{Yamauchi-Primary}, which generalizes the $1$-point modularity result of Dong, Li, and Mason \cite{DLM-Orbifold} to $n$-point functions. When applied to our situation, this result states there are scalars $A_{\nu ,k}^\gamma$  for each $\gamma =\left(\begin{smallmatrix} a&b\\c&d \end{smallmatrix}\right) \in \text{SL}_2 (\mathbb{Z})$ such that for any $v\in V_{[\wt [v]]}$ we have
 \begin{equation}
 \begin{aligned}
 &S^{\nu} \left(\psi ;\frac{z_1}{c\tau +d} ,\dots ,\frac{z_n}{c\tau +d} ,\left\{v,\frac{x}{c\tau +d}\right\},\frac{a\tau +b}{c\tau +d} \right) \\
 &\hspace{15mm}=(c\tau +d)^{\wt [v]+n} \sum_{k=1}^N A_{\nu ,k}^\gamma  S^{k} \left(\psi ;z_1 ,\dots ,z_n ,\{v,x\},\tau ,z \right).
  \end{aligned} \label{ZhuTheorem}
 \end{equation}
 Throughout the remainder of Section \ref{SectionProofs}, for $\gamma = \left( \begin{smallmatrix} a&b \\c&d \end{smallmatrix} \right) \in \SL_2 (\mathbb{Z})$ and $\tau \in \mathbb{H}$ we set
 \[
 \gamma \tau := \frac{a\tau +b}{c\tau +d}, \qquad
 \hat{\gamma}^\tau := a\tau +b , \qquad \text{and} \qquad
 \check{\gamma}_\tau := c\tau +d.
 \]
 As we fix $\gamma$, we will typically write $A_k^\nu$ in place of $A_{\nu ,k}^\gamma$. \\
 \indent For a subset $U$ of a set $W$, let
 \begin{align}
 D_{\sigma} &:= \prod_{j<\sigma (j)} \langle v_j ,v_{\sigma (j)}\rangle {\check{\gamma}_{\tau}}^2  P_2 \left( z_{\sigma (j),j} ,\tau \right), \notag \\
 E_{\sigma}&:=\prod_{j<\sigma (j)} \langle v_j ,v_{\sigma (j)} \rangle \left({\check{\gamma}_{\tau}}^2 P_2 \left( z_{\sigma (j),j} ,\tau \right) -\frac{c{\check{\gamma}_{\tau }}}{2\pi i}\right), \notag \\
 F_{U}^{W} &:= \prod_{s\in W\setminus U} {\check{\gamma}_{\tau}} \sum_{m_s \geq 1} P_{m_1}(x_s,\tau)v_s[m_s-1], \notag \\
 G_{U}^{W}&:= \prod_{s\in W\setminus U} \sum_{m_s \geq 1} P_{m_s}\left(\frac{x_s}{\check{\gamma}_\tau}, \gamma \tau\right)v_s [m_s -1], \quad \text{and} \notag \\
 H_{U}^{W}&:= \prod_{s\in W\setminus U} \left(A_\gamma (x_s,\tau)v_s [0] -\frac{c}{2\pi i}v_s[1] + {\check{\gamma}_{\tau}}\sum_{m_s \geq 1} P_{m_s}\left(x_s, \tau \right)v_s [m_s -1]\right), \notag
 \end{align}
 where 
 \[
 A_\gamma (x_s,\tau) := \frac{\check{\gamma}_\tau}{2} -cx_s -\frac{1}{2}.
 \]
 For $j \geq 1$ and a nested set of subsets 
  \[
  U_\ell \subseteq f(\sigma_{\ell -1}) \subseteq U_{\ell -1} \subseteq \cdots f(\sigma_2)\subseteq U_2 \subseteq f(\sigma_1)\subseteq U_1 =U_0 =\{1,\dots ,n\},
  \]
   set $X_{U_j}^{\sigma_{j-1}} := X_{U_{j}}^{f(\sigma_{j-1})}$ for $X=F,G,H$. In this notation we also set $H_{U_{1}}^{\sigma_0} =G_{U_{1}}^{\sigma_0} =F_{U_{1}}^{\sigma_0}:=1$. By (\ref{commutation}), it follows that $F_{U_i}^{\sigma_{i-1}}F_{U_{j}}^{\sigma_{j-1}} =F_{U_j}^{\sigma_{j-1}} F_{U_i}^{\sigma_{i-1}}$, $G_{U_i}^{\sigma_{i-1}} G_{U_{j}}^{\sigma_{j-1}} =G_{U_j}^{\sigma_{j-1}} G_{U_i}^{\sigma_{i-1}}$, and $H_{U_i}^{\sigma_{i-1}} H_{U_{j}}^{\sigma_{j-1}} =H_{U_j}^{\sigma_{j-1}} H_{U_i}^{\sigma_{i-1}}$ for any $1\leq i,j \leq \ell$. Moreover, let $U_j^{j -1} := f(\sigma_{j -1})\setminus U_j$ denote the complement of $U_j$ in $f(\sigma_{j -1})$, and $\lvert U_j\rvert$ be the number of elements in $U_j$.
   
 We are now in position to establish the necessary lemmas.
 
\begin{lemma}\label{LemmaA}
 Let $t_1 ,\dots ,t_{\lvert U_\ell \rvert}$ denote the elements of $U_{\ell}$. Then for every $U_\ell$, $\ell \geq 1$, we have
 \begin{align}
 S^{\nu} &\left(1; \frac{z_{t_1}}{\check{\gamma}_{\tau}}, \dots ,\frac{z_{t_{\lvert U_\ell \rvert}}}{\check{\gamma}_{\tau}}, \left\{ G_{U_\ell}^{\sigma_{\ell -1}} G_{U_{\ell -1}}^{\sigma_{\ell -2}}\cdots G_{U_2}^{\sigma_{1}}v, \frac{x}{\check{\gamma}_{\tau}}\right\}, \gamma \tau \right) \notag \\
 &=\check{\gamma}_{\tau}^{\lvert U_\ell \rvert +\wt [v]} \sum_{k=1}^N A_{k}^{\nu} S^k \left(1; z_{t_1}, \dots , z_{t_{\lvert U_\ell \rvert}}, \left\{ H_{U_\ell}^{\sigma_{\ell -1}} H_{U_{\ell -1}}^{\sigma_{\ell -2}}\cdots H_{U_2}^{\sigma_{1}}v, x\right\}, \tau \right). \notag
 \end{align}
\end{lemma}

\begin{proof}  
 We first note that 
 \begin{align}
 &S^{\nu} \left(1; \frac{z_{t_1}}{\check{\gamma}_{\tau}}, \dots ,\frac{z_{t_{\lvert U_\ell \rvert}}}{\check{\gamma}_{\tau}}, \left\{ G_{U_\ell}^{\sigma_{\ell -1}} G_{U_{\ell -1}}^{\sigma_{\ell -2}}\cdots G_{U_2}^{\sigma_{1}}v, \frac{x}{\check{\gamma}_{\tau}}\right\}, \gamma \tau \right) \notag \\
 &=S^{\nu} \Biggl (1; \frac{z_{t_1}}{\check{\gamma}_{\tau}}, \dots ,\frac{z_{t_{\lvert U_\ell \rvert}}}{\check{\gamma}_{\tau}}, \Biggl\{ \Biggl (\prod_{i=2}^{\ell} \prod_{s_{i} \in U_{i}^{i-1}} \sum_{m_{s_{i}} \geq 1} P_{m_{s_{i}}} \left(\frac{x_{s_{i}}}{{\check{\gamma}_{\tau}}}, \gamma \tau \right)v_{s_{i}}[m_{s_{i}}-1] \Biggl ) v, \frac{x}{\check{\gamma}_{\tau}} \Biggl \}, \gamma \tau \Biggl ) \notag \\
 &=S^{\nu} \Biggl (1; \frac{z_{t_1}}{\check{\gamma}_{\tau}}, \dots ,\frac{z_{t_{\lvert U_\ell \rvert}}}{\check{\gamma}_{\tau}}, \Biggl\{ \Biggl (\prod_{i=2}^{\ell} \prod_{s_{i} \in U_{i}^{i-1}} \Biggl (\left(\frac{{\check{\gamma}_{\tau}}}{2} - c(x-z_{s_i}) -\frac{1}{2}\right) v_{s_{i}}[0] \notag \\
 &\hspace{15mm}  -\frac{c{\check{\gamma}_{\tau}}}{2\pi i}v_{s_{i}}[1] + \sum_{m_{s_{i}} \geq 1} {\check{\gamma}_{\tau}}^{m_{s_{i}}} P_{m_{s_{i}}} (x_{s_{i}}, \tau )v_{s_{i}}[m_{s_{i}}-1] \Biggl ) v, \frac{x}{\check{\gamma}_{\tau}} \Biggl \}, \gamma \tau \Biggl ), \label{rando1}
 \end{align}
 where the index $i$ begins at $2$ since $G_{U_1}^{\sigma_{0}} =1$. Since the functions $S^k$ are linear and each component being summed is comprised of an element of weight (cf.\ 2.\ of Lemma \ref{LemWt1})
 \begin{align*}
 &\wt \left[ \prod_{i=2}^\ell \left(\prod_{s_{i} \in U_{i}^{i-1}} v_{s_{i}} [m_{s_{i}}-1] \right)v \right]
 =\wt [v] +\sum_{i=2}^\ell \left\lvert U_i^{i-1} \right\rvert -\sum_{i=2}^{\ell} \sum_{s_i \in U_i^{i-1}} m_{s_{i}},
 \end{align*}
 it follows from (\ref{ZhuTheorem}) that (\ref{rando1}) becomes
 \begin{align}
 {\check{\gamma}_{\tau}}^{\lvert U_\ell \rvert +\wt [v]} \sum_{k=1}^N A_{k}^{\nu} S^k \left(1; z_{t_1}, \dots , z_{t_{\lvert U_\ell \rvert}}, \left\{ H_{U_\ell}^{\sigma_{\ell -1}} H_{U_{\ell -1}}^{\sigma_{\ell -2}}\cdots H_{U_2}^{\sigma_{1}}v, x\right\}, \tau \right). \notag
 \end{align}
 This completes the proof.
\end{proof}

\begin{lemma}\label{LemmaB}
 For $U_{\ell -1}$ such that $\lvert U_{\ell -1} \rvert \geq 0$, we have
 \begin{align}
 &S^{\nu} \left( \prod_{r\in U_{\ell -1}} o(v_r) ; \left\{ G_{U_{\ell -1}}^{\sigma_{\ell -2}} \cdots G_{U_2}^{\sigma_{1}} v,\frac{x}{\check{\gamma}_{\tau}}\right\}, \gamma \tau \right) \notag \\
 &={\check{\gamma}_{\tau}}^{\wt [v]} \sum_{k=1}^N A_{k}^{\nu} \sum_{\sigma_{\ell -1} \in I(U_{\ell -1})} D_{\sigma_{\ell -1}} \sum_{U_\ell \subseteq f(\sigma_{\ell -1})} S^k \left( \prod_{r\in U_{\ell}}{\check{\gamma}_{\tau}}o(v_r); \left\{ F_{U_{\ell}}^{\sigma_{\ell -1}} H_{U_{\ell -1}}^{\sigma_{\ell -2}} \cdots H_{U_2}^{\sigma_{1}}v,x \right\} ,\tau \right) \notag \\
 &\hspace{5mm} -\sum_{\sigma_{\ell -1} \in I(U_{\ell -1})} E_{\sigma_{\ell -1}} \sum_{\substack{U_\ell \subseteq f(\sigma_{\ell -1}) \\ U_\ell \not = U_{\ell -1}}} S^{\nu} \left( \prod_{r\in U_{\ell}} o(v_r); \left\{G_{U_{\ell}}^{\sigma_{\ell -1}} G_{U_{\ell -1}}^{\sigma_{\ell -2}} \cdots G_{U_2}^{\sigma_{1}}v, \frac{x}{\check{\gamma}_{\tau}} \right\} , \gamma \tau \right). \notag
 \end{align}
\end{lemma}

\begin{proof} 
 Utilizing Proposition \ref{Prop1} twice, with a use of Lemma \ref{LemmaA} in between, we find
 \begin{align}
 &\sum_{\sigma_{\ell -1} \in I(U_{\ell -1})} \prod_{j<\sigma_{\ell -1} (j)} \langle v_j ,v_{\sigma_{\ell -1} (j)} \rangle P_2 \left(\frac{z_{\sigma_{\ell -1} (j),j}}{{\check{\gamma}_\tau}} ,\gamma \tau \right) \notag \\
 &\hspace{5mm} \times \sum_{U_\ell \subseteq f(\sigma_{\ell -1})} S^{\nu} \left(\prod_{r\in U_\ell} o(v_r); \left\{ \left( \prod_{s\in U_{\ell}^{\ell -1}} \phi \left(v_s ,\frac{x_s}{\check{\gamma}_{\tau}}, \gamma \tau \right) \right)G_{U_{\ell -1}}^{\sigma_{\ell -2}}\cdots G_{U_2}^{\sigma_{1}}v,\frac{x}{\check{\gamma}_{\tau}} \right\}, \gamma \tau \right) \notag \\
 &\hspace{10mm}= S^{\nu} \left(1; \frac{z_{t_1}}{\check{\gamma}_{\tau}}, \dots ,\frac{z_{t_{\lvert U_{\ell -1} \rvert}}}{\check{\gamma}_{\tau}}, \left\{G_{U_{\ell -1}}^{\sigma_{\ell -2}}\cdots G_{U_2}^{\sigma_{1}} v, \frac{x}{\check{\gamma}_{\tau}} \right\}, \gamma \tau \right) \notag \\
 &\hspace{10mm}={\check{\gamma}_{\tau}}^{\lvert U_{\ell -1} \rvert +\wt [v]} \sum_{k=1}^N A_{k}^{\nu} S^k \left(1; z_{t_1}, \dots ,z_{t_{\lvert U_{\ell -1} \rvert}}, \left\{H_{U_{\ell -1}}^{\sigma_{\ell -2}} \cdots H_{U_2}^{\sigma_{1}} v,x\right\} ,\tau \right) \notag \\
 &\hspace{10mm}={\check{\gamma}_{\tau}}^{\lvert U_{\ell -1} \rvert +\wt [v]}\sum_{k=1}^N A_{k}^{\nu} \sum_{\sigma_{\ell -1} \in I(U_{\ell -1})} \prod_{j<\sigma_{\ell -1} (j)} \langle v_j ,v_{\sigma_{\ell -1} (j)} \rangle P_2 \left( z_{\sigma_{\ell -1} (j),j} ,\tau \right) \notag \\
 &\hspace{20mm} \times \sum_{U_\ell\subseteq f(\sigma_{\ell -1})} S^k \left( \prod_{r\in U_\ell} o(v_r); \left\{ \left(\prod_{s\in U_{\ell}^{\ell -1}} \phi (v_s)\right) H_{U_{\ell -1}}^{\sigma_{\ell -2}}\cdots H_{U_2}^{\sigma_{1}} v,x\right\} ,\tau \right). \label{hmm1}
 \end{align}
 Isolating the piece associated to $\sigma_{\ell -1} =1$ and $U_\ell =U_{\ell -1}$ (so that $U_{\ell}^{\ell -1} =\emptyset$) in the left side of (\ref{hmm1}),  we have
 \begin{align}
 &S^{\nu} \left(\prod_{r \in U_{\ell -1}} o(v_r) ;\left\{G_{U_{\ell -1}}^{\sigma_{\ell -2}} \cdots G_{U_2}^{\sigma_{1}}v,\frac{x}{\check{\gamma}_{\tau}}\right\} ,\gamma \tau \right) \notag \\
 &= {\check{\gamma}_{\tau}}^{\vert U_{\ell -1} \rvert +\wt [v]}\sum_{k=1}^N A_{k}^{\nu} \sum_{\sigma_{\ell -1} \in I(U_{\ell -1})} \prod_{j<\sigma_{\ell -1} (j)} \langle v_j ,v_{\sigma_{\ell -1} (j)} \rangle P_2 \left( z_{\sigma_{\ell -1} (j),j} ,\tau \right) \notag \\
 &\hspace{5mm} \times \sum_{U_\ell \subseteq f(\sigma_{\ell -1})} S^k \left( \prod_{r\in U_\ell} o(v_r); \left\{ \left(\prod_{s\in U_{\ell}^{\ell -1}} \phi (v_s)\right) H_{U_{\ell -1}}^{\sigma_{\ell -2}} \cdots H_{U_2}^{\sigma_{1}} v,x\right\} ,\tau \right) \notag \\
 & \hspace{7mm} -\sum_{\sigma_{\ell -1} \in I(U_{\ell -1})} \prod_{j<\sigma_{\ell -1} (j)} \langle v_j ,v_{\sigma (j)} \rangle P_2 \left( \frac{z_{\sigma_{\ell -1} (j),j}}{{\check{\gamma}_\tau}} ,\gamma \tau \right) \notag \\
 &\hspace{5mm} \times \sum_{\substack{U_\ell \subseteq f(\sigma_{\ell -1}) \\ U_\ell \not = U_{\ell -1}}} 
 S^{\nu} \left( \prod_{r\in U_\ell} o(v_r); \left\{ \left(\prod_{s\in U_{\ell}^{\ell -1}} \phi \left(v_s ,\frac{x_s}{{\check{\gamma}_\tau}} ,\gamma \tau \right)\right) G_{U_{\ell -1}}^{\sigma_{\ell -2}} \cdots G_{U_2}^{\sigma_{1}} v,\frac{x}{\check{\gamma}_{\tau}}\right\} ,\gamma \tau \right) \notag \\
 &={\check{\gamma}_{\tau}}^{\wt [v]}\sum_{k=1}^N A_{k}^{\nu} \sum_{\sigma_{\ell -1} \in I(U_{\ell -1})} D_{\sigma_{\ell -1}} \sum_{U_\ell \subseteq f(\sigma_{\ell -1})} S^k \left( \prod_{r\in U_\ell} {\check {\gamma }_{\tau }}o(v_r); \left\{F_{U_\ell}^{\sigma_{\ell -1}} H_{U_{\ell -1}}^{\sigma_{\ell -2}} \cdots H_{U_2}^{\sigma_{1}}v,x\right\} ,\tau \right) \notag \\
 &\hspace{7mm} - \sum_{\sigma_{\ell -1} \in I(U_{\ell -1})} E_{\sigma_{\ell -1}} \sum_{\substack{U_\ell \subseteq f(\sigma_{\ell -1}) \\ U_\ell \not = U_{\ell -1}}} S^{\nu} \left( \prod_{r\in U_\ell} o(v_r) ; \left\{ G_{U_\ell}^{\sigma_{\ell -1}}G_{U_{\ell -1}}^{\sigma_{\ell -2}} \cdots G_{U_2}^{\sigma_{1}}v, \frac{x}{\check{\gamma}_{\tau}}\right\},\gamma \tau \right), \notag
 \end{align}
 where we used that $\lvert m(\sigma_{\ell -1})\rvert +\lvert U_\ell \rvert +\lvert U_{\ell}^{\ell -1} \rvert =\lvert U_{\ell -1} \rvert$ for each $\sigma_{\ell -1}$.
 \end{proof}
 
 Note that $E_{\sigma_{\ell -1}}$ and $D_{\sigma_{\ell -1}}$ are both $1$ when $\lvert U_{\ell -1} \rvert =0,1$.


\subsection{Step three\label{Step3}}
 
 The third step uses the symmetry of the equations developed in the previous steps to show that every term with a function $P_k$ vanishes, while terms only involving the modular anomalies of $P_1$ and $P_2$ remain. To help accomplish this, we make note of a result in \cite[Lemma $4.1$]{Miy-Theta}.

\begin{lemma}\label{LemmaC}
 Suppose $\lvert m(\sigma)\rvert =2p$. Then
 \[
 \sum_{\sigma_1 +\cdots +\sigma_t =\sigma} (-1)^t E_{\sigma_1} E_{\sigma_2} \cdots E_{\sigma_t} =(-1)^p E_{\sigma}.
 \]
 \hfill \qed
\end{lemma}

 We may now prove the following lemma.
 
\begin{lemma}\label{LemmaD}
 We have
 \begin{align*}
 &\sum_{\sigma \in I(n)}\sum_{\substack{ \sigma_1 ,\sigma_2 \in I(n) \\ \sigma_1 +\sigma_{2}=\sigma}} (-1)^{\frac{\lvert m(\sigma_1)\rvert}{2}}E_{\sigma_1} D_{\sigma_{2}} \sum_{U\subseteq W\subseteq f(\sigma)} (-1)^{\lvert f(\sigma) \setminus W\rvert} S^k \left( \prod_{r\in U} {\check {\gamma }_{\tau }}o(v_r); \left\{F_{U}^{W}H_{W}^{\sigma}v,x\right\} ,\tau \right) \notag \\
 &= \sum_{\sigma \in I(n)} \left( \prod_{j<\sigma (j)} \left( \frac{c{\check{\gamma}_{\tau}}\langle v_j ,v_{\sigma (j)} \rangle }{2\pi i} \right) \right)\sum_{U\subseteq f(\sigma)} S^k \left( \prod_{r\in U} {\check{\gamma}_{\tau}}o(v_r); \left\{ \left( \prod_{s\in f(\sigma) \setminus U} -B_s \right) v, x \right\} ,\tau \right), \notag
 \end{align*}
 where $B_s = B(v_s ,x_s) := A_\gamma (x_s ,\tau ) v_s [0] -\frac{c}{2\pi i} v_s [1]$.
\end{lemma}

\begin{proof}
By linearity, we can break the proof into two parts.
Set $Q_{\sigma} := \prod_{j <\sigma (j)} \left(-\frac{c{\check{\gamma}_{\tau}}}{2\pi i} \right)$ and $R_{\sigma} := \prod_{j <\sigma (j)} \check{\gamma}_{\tau} P_2 \left( z_{\sigma (j),j} ,\tau \right)$. For $\sigma ,\sigma_1 ,\sigma_2 \in I(n)$ such that $\sigma_1 +\sigma_2 =\sigma$, we may consider decompositions of $\sigma_1$ as $\sigma_3 +\sigma_4 =\sigma_1$ with $\sigma_3 ,\sigma_4 \in I(n)$, so that
\begin{align}
 &\sum_{\substack{\sigma_1 ,\sigma_2 \in I(n) \\\sigma_1 +\sigma_2 =\sigma}} (-1)^{\frac{\lvert m(\sigma_1)\rvert}{2}} E_{\sigma_1} D_{\sigma_2} = \sum_{\sigma_2 +\sigma_3 +\sigma_4 =\sigma} (-1)^{\frac{\lvert m(\sigma_3)\rvert}{2}} (-1)^{\frac{\lvert m(\sigma_4) \rvert}{2}} Q_{\sigma_3} R_{\sigma_4} R_{\sigma_2} \prod_{j<\sigma (j)} \langle v_j ,v_{\sigma (j)} \rangle \notag \\
 &\hspace{15mm}= \sum_{\sigma_3 +\sigma' =\sigma} (-1)^{\frac{\lvert m(\sigma_3)\rvert}{2}} Q_{\sigma_3} \prod_{j<\sigma (j)} \langle v_j ,v_{\sigma (j)} \rangle \sum_{\sigma_2 +\sigma_4 =\sigma'} (-1)^{\frac{\lvert m(\sigma_4) \rvert}{2}} R_{\sigma_2 +\sigma_4}. \label{quack1}
\end{align}
 Suppose $m(\sigma') \not = \emptyset$. For $\sigma' = (s_1 ,s_2)\cdots (s_{2\ell -1} ,s_{2\ell})$, we have there are $\binom{\ell}{r}$ possible many $\sigma_4$ with $\sigma_2 +\sigma_4 =\sigma'$ such that $\lvert m(\sigma_4) \rvert =2r$. For such $\sigma'$, we have
 \[
 \sum_{\sigma_2 +\sigma_4 =\sigma'} (-1)^{\frac{\lvert m(\sigma_4) \rvert}{2}} R_{\sigma_2 +\sigma_4} = \sum_{r=0}^\ell (-1)^r \binom{\ell}{r} R_{\sigma'} =0
 \]
 since $\sum_{r=0}^\ell (-1)^r \binom{\ell}{r} =0$ for $\ell >0$. In the case $m(\sigma')=\emptyset$, continuing the calculation in \eqref{quack1} gives
 \begin{align*}
 &\sum_{\substack{\sigma_1 ,\sigma_2 \in I(n) \\\sigma_1 +\sigma_2 =\sigma}} (-1)^{\frac{\lvert m(\sigma_1)\rvert}{2}} E_{\sigma_1} D_{\sigma_2}
 = \sum_{\sigma \in I(n)} (-1)^{\lvert m(\sigma)\rvert /2} \prod_{j <\sigma (j)} \left(-\frac{c{\check{\gamma}_{\tau}}}{2\pi i} \right) \prod_{j<\sigma (j)} \langle v_j ,v_{\sigma (j)} \rangle \notag \\
 &\hspace{15mm}= \sum_{\sigma \in I(n)} (-1)^{\lvert m(\sigma)\rvert } \prod_{j <\sigma (j)} \left(\frac{c{\check{\gamma}_{\tau}}\langle v_j ,v_{\sigma (j)} \rangle}{2\pi i} \right)= \sum_{\sigma \in I(n)} \prod_{j <\sigma (j)} \left(\frac{c{\check{\gamma}_{\tau}}\langle v_j ,v_{\sigma (j)} \rangle}{2\pi i} \right). 
 \end{align*}
 This proves the first part.\\
 \indent For the second part, set $\widehat{\phi} \left( v_s\right):= {\check{\gamma}_{\tau}} \phi \left( v_s ,x_s ,\tau \right)$. Fix $\sigma$ and consider
 \[
 \sum_{U\subseteq W \subseteq f(\sigma)} (-1)^{\lvert f(\sigma)\setminus W \rvert} F_U^W H_W^{\sigma}.
 \] 
 Note that $\left(f(\sigma)\setminus W\right) \sqcup \left(W\setminus U\right) =f(\sigma)\setminus U$. Then
 \begin{align}
 &\sum_{U\subseteq W \subseteq f(\sigma)} (-1)^{\lvert f(\sigma)\setminus W \rvert} F_U^W H_W^{\sigma} =\sum_{U\subseteq W \subseteq f(\sigma)} (-1)^{\lvert f(\sigma)\setminus W \rvert} \prod_{r\in W\setminus U} \widehat{\phi}\left(v_r \right) \prod_{s\in f(\sigma)\setminus W} \left(B_s +\widehat{\phi}\left(v_s \right)\right) \notag \\
 &\hspace{15mm}=\sum_{U\subseteq W \subseteq f(\sigma)} (-1)^{\lvert f(\sigma)\setminus W \rvert}\prod_{r\in W\setminus U} \widehat{\phi}\left(v_r \right) \sum_{\substack{X_1 ,X_3 \subseteq f(\sigma)\setminus W \\ X_1 \sqcup X_3 =f(\sigma)\setminus W}} \prod_{s_1 \in X_1} B_{s_1} \prod_{s_2 \in X_3} \widehat{\phi} \left( v_{s_2}\right) \notag \\
 &\hspace{15mm}=\sum_{U\subseteq W \subseteq f(\sigma)} (-1)^{\lvert f(\sigma)\setminus W \rvert} \sum_{\substack{X_1 ,X_3 \subseteq f(\sigma)\setminus W \\ X_1 \sqcup X_3 =f(\sigma)\setminus W}} \prod_{s \in X_1} B_{s} \prod_{r \in X_3\sqcup W\setminus U} \widehat{\phi} \left( v_{r}\right) \notag \\
 &\hspace{15mm}=\sum_{U\subseteq W \subseteq f(\sigma)} (-1)^{\lvert f(\sigma)\setminus W \rvert} \sum_{\substack{X_1 \subseteq f(\sigma) \setminus W ,X_2 \subseteq f(\sigma)\setminus U \\ X_1 \sqcup X_2 =f(\sigma)\setminus U}} \prod_{s \in X_1} B_{s} \prod_{r \in X_2} \widehat{\phi} \left( v_{r}\right). \label{tort1}
 \end{align}
 For fixed $X_1 \subseteq f(\sigma)\setminus W \subseteq f(\sigma) \setminus U$, the product $\prod_{r \in X_2} \widehat{\phi} \left( v_{r}\right)$ is the same regardless of the $W$, so long as $U\subseteq W$ and $X_1 \subseteq f(\sigma) \setminus W$. Therefore, we are interested in counting how many ways we can choose $W$ so that $U\subseteq W$ and $X_1 \subseteq f(\sigma) \setminus W$. Since there are $\lvert f(\sigma)\rvert - \lvert X_1 \rvert -\lvert U \rvert =\lvert f(\sigma) \setminus U\rvert - \lvert X_1 \rvert$ many elements $W$ may contain which are not in $U$, there are $\binom{\lvert f(\sigma) \setminus U\rvert - \lvert X_1 \rvert}{j}$ many ways to choose $W$ so that $\lvert W \rvert = \lvert U \rvert +j$. Then $(-1)^{\lvert f(\sigma)\setminus W\rvert} = (-1)^{\lvert f(\sigma) \setminus U\rvert} (-1)^j$, and for the fixed $X_1$ we find the total number of $\prod_{s \in X_1} B_{s} \prod_{r \in X_2} \widehat{\phi} \left( v_{r}\right)$ terms in (\ref{tort1}) is
 \begin{align*}
 &(-1)^{\lvert f(\sigma)\setminus U \rvert} \sum_{j=0}^{\lvert f(\sigma)\setminus U\rvert -\lvert X_1 \rvert} (-1)^j \binom{\lvert f(\sigma)\setminus U\rvert -\lvert X_1 \rvert}{j} .
 \end{align*}
 However, this sum equals $0$ so long as $\lvert f(\sigma)\setminus U \rvert -\lvert X_1 \rvert >0$. In the case $\lvert f(\sigma)\setminus U \rvert -\lvert X_1 \rvert =0$, we have $X_1 = f(\sigma)\setminus U$, and continuing the calculation in (\ref{tort1}) we find
\begin{align}
 &\sum_{U\subseteq W \subseteq f(\sigma)} (-1)^{\lvert f(\sigma)\setminus W \rvert} F_U^W H_W^{\sigma} = \sum_{U\subseteq W \subseteq f(\sigma)} (-1)^{\lvert f(\sigma)\setminus W \rvert} \sum_{\substack{X_1 \subseteq f(\sigma) \setminus W ,X_2 \subseteq f(\sigma)\setminus U \\ X_1 \sqcup X_2 =f(\sigma)\setminus U}} \prod_{s \in X_1} B_{s} \prod_{r \in X_2} \widehat{\phi} \left( v_{r}\right) \notag \\
 &\hspace{15mm}=\sum_{U\subseteq f(\sigma)} (-1)^{\lvert f(\sigma)\setminus U \rvert} \prod_{s\in f(\sigma)\setminus U} B_s 
 = \sum_{U\subseteq f(\sigma)}\prod_{s\in f(\sigma)\setminus U} \left( -B_s \right). \notag
 \end{align}
 This completes the proof.
 \end{proof} 
 
 
\subsection{Step four\label{Step4}}

 Here we combine the previous steps to obtain the modular transformation properties for the functions of an isolated product of $n$ zero modes.
 
\begin{proposition} \label{HardThm1}
 With the notation and assumptions above, along with requiring $v_{s}(0)v=0$ for all $1\leq s\leq n$, we have
 \begin{align}
 S^{\nu} \left(\prod_{r=1}^n o(v_r) ;\left\{v,\frac{x}{\check{\gamma}_{\tau}}\right\} ,\gamma \tau\right) =&{\check{\gamma}_{\tau}}^{\wt [v]}\sum_{k=1}^N A_{k}^{\nu} \sum_{\sigma \in I(n)} \left(\prod_{j<\sigma (j)} \left(\frac{c{\check{\gamma}_{\tau}}\langle v_j ,v_{\sigma (j)}\rangle}{2\pi i} \right)\right)\notag \\
 &\times\sum_{U\subseteq f(\sigma)} S^k \left( \prod_{r\in U} {\check{\gamma}_{\tau}} o(v_r) ; \left\{ \left(\prod_{s\in f(\sigma)\setminus U} \frac{c}{2\pi i}v_s [1]\right) v,x\right\}, \tau \right). \notag
 \end{align}
\end{proposition}

\begin{proof}
Using Lemma \ref{LemmaB} applied to $U_1 =\{1,\dots ,n\}$, we find
 \begin{align}
 &S^{\nu} \left(\prod_{r=1}^n o(v_r) ;\left\{v,\frac{x}{\check{\gamma}_{\tau}}\right\} ,\gamma \tau\right) \notag \\
 &={\check{\gamma}_{\tau}}^{\wt [v]}\sum_{k=1}^N A_{k}^{\nu} \sum_{\sigma_1 \in I(U_1)} D_{\sigma_1} \sum_{U_2 \subseteq f(\sigma_1)} S^k \left( \prod_{r\in U_2} {\check {\gamma }_{\tau }}o(v_r); \left\{F_{U_2}^{\sigma_1}v,x\right\} ,\tau \right) \notag \\
 &\hspace{10mm} - \sum_{\sigma_1 \in I(U_1)} E_{\sigma_1} \sum_{U_1 \not =U_2 \subseteq f(\sigma_1)} S^{\nu} \left( \prod_{r\in U_2} o(v_r) ; \left\{ G_{U_2}^{\sigma_1}v, \frac{x}{\check{\gamma}_{\tau}}\right\},\gamma \tau\right). \label{dog1}
 \end{align}
 Let $\ell \geq 1$ be the length of the largest chain of proper subsets
\begin{equation}
 \emptyset =U_\ell \subset U_{\ell -1} \subset \cdots \subset U_1 =\{ 1,\dots ,n\} \label{nested}
\end{equation}
occurring after successive applications of Lemma \ref{LemmaB}. Then reapplying Lemma \ref{LemmaB} to the last term in (\ref{dog1}), and repeating this process on the corresponding term until we reach $U_\ell$, we obtain
 \begin{align}
 &S^{\nu} \left(\prod_{r=1}^n o(v_r) ;\left\{v,\frac{x}{\check{\gamma}_{\tau}}\right\} ,\gamma \tau\right) \notag \\
 &=\sum_{p=1}^{\ell -1} {\check{\gamma}_{\tau}}^{\wt [v]}\sum_{k=1}^N A_{k}^{\nu} \sum_{\sigma_1 \in I(U_1)} \sum_{\substack{U_2 \subseteq f(\sigma_1) \\ U_2 \not = U_1}} \cdots \sum_{\sigma_{p-1} \in I(U_{p-1})} \sum_{\substack{U_{p} \subseteq f(\sigma_{p-1}) \\ U_{p} \not = U_{p-1}}} \sum_{\sigma_{p} \in I(U_{p})} \sum_{U_{p+1} \subseteq f(\sigma_{p})} \notag \\
 &\hspace{15mm} \times (-1)^{p-1} E_{\sigma_1} \cdots E_{\sigma_{p-1}} D_{\sigma_{p}} S^k \left( \prod_{r\in U_{p+1}} {\check {\gamma }_{\tau }}o(v_r); \left\{F_{U_{p+1}}^{\sigma_{p}}H_{U_{p}}^{\sigma_{p-1}}\cdots H_{U_2}^{\sigma_{1}}v,x\right\} ,\tau \right) \notag \\
 &\hspace{5mm} +{\check{\gamma}_{\tau}}^{\wt [v]} \sum_{k=1}^N A_{k}^{\nu} \sum_{\sigma_1 \in I(U_1)} \sum_{\substack{U_2 \subseteq f(\sigma_1) \\ U_2 \not = U_1}} \cdots \sum_{\sigma_{\ell -2} \in I(U_{\ell -2})} \sum_{\substack{U_{\ell -1} \subseteq f(\sigma_{\ell -2}) \\ U_{\ell -1} \not = U_{\ell -2}}} \sum_{\sigma_{\ell -1} \in I(U_{\ell -1})} \sum_{\substack{\emptyset =U_{\ell} \subseteq f(\sigma_{\ell -1}) \\U_{\ell} \not = U_{\ell -1}}} \notag \\
 &\hspace{15mm} \times (-1)^{\ell -1} E_{\sigma_1} \cdots E_{\sigma_{\ell -2}} E_{\sigma_{\ell -1}} S^{k} \left(1; \left\{H_{U_\ell}^{\sigma_{\ell -1}}H_{U_{\ell -1}}^{\sigma_{\ell -2}}\cdots H_{U_2}^{\sigma_{1}}v, x\right\}, \tau \right), \label{cat1}
 \end{align}
 where the last equality also uses Lemma \ref{LemmaA}.\\
\indent There may be $\sigma_j =1$ in (\ref{cat1}). In this case, the condition $U_{j+1} \not =U_j$ is needed in the sum to ensure an iteration of Lemma \ref{LemmaB} on the last term. Otherwise, $U_{j+1} \not = U_j$ by default. Since $\sigma_j \in I(U_j) \subseteq f(\sigma_{j-1})$, we have $m(\sigma_1) \cap \cdots \cap m(\sigma_{p-1}) =\emptyset$. Therefore, $\sigma_1 +\cdots +\sigma_{p-1} =\sigma'$ for some $\sigma' \in I(n)$. In fact, for any $\sigma' \in I(n)$ there are $\widetilde{\sigma}_1 , \dots ,\widetilde{\sigma}_{p-1} \in I(n)$ satisfying the conditions in (\ref{cat1}) such that $\widetilde{\sigma}_1 + \cdots +\widetilde{\sigma}_{p-1}=\sigma'$. Meanwhile, $\sigma'$ may also be written as $\sigma' =\psi_1 +\cdots +\psi_i$ for some $\psi_1 ,\dots ,\psi_i \in I(n)$ with no additional conditions. Such a reformulation not only affects the construction of $U_2 , \dots , U_{p}$, but $H_{U_{p}}^{\sigma_{p-1}} \cdots H_{U_2}^{\sigma_1}$ as well.\\
\indent We want to reformulate (\ref{cat1}) in terms of general involutions $\psi_1 , \dots ,\psi_i$ and sets $U\subseteq W \subseteq f(\sigma')$ which accomplish the same expression as the $\sigma_1 ,\dots ,\sigma_{\ell -1}$ and $U_1 ,\dots ,U_{\ell -1}$ that form the chain (\ref{nested}), but without the nested restrictions.\\ 
 \indent For starters, by Lemma \ref{LemmaC} we have
 \[
 \sum_{\psi_1 +\cdots + \psi_i =\sigma'} (-1)^i E_{\psi_1} \cdots E_{\psi_i} =(-1)^{\frac{\lvert m(\sigma')\rvert}{2}} E_{\sigma'}.
 \] 
 Set $U:=U_{p+1}$ and take $W$ to be the set $W\subseteq f(\sigma')$ such that
 \[
 f(\sigma') \setminus W =\bigsqcup_{j=1}^{p-1} f(\sigma_j)\setminus U_{j+1}.
 \]
 Note $U\subseteq W\subseteq f(\sigma')$ since $\sigma' \in I(n)$. On one hand, this shows $H_{U_{p}}^{\sigma_{p-1}} \cdots H_{U_2}^{\sigma_1} =H_W^{\sigma'}$. On the other, there may be many ways to obtain this equality resulting from different Lemma \ref{LemmaB} iterations. Note, however, that the maximal number of elements to choose such a fixed $W \subseteq f(\sigma')$ in this way is $\lvert f(\sigma')\setminus W \rvert$. The number of ways to choose this corresponding to $j$ iterations of Lemma \ref{LemmaB} is $\binom{\lvert f(\sigma')\setminus W\rvert}{j}$.\\ 
 \indent Both the choices of $\sigma_1 ,\dots ,\sigma_p$ and $U_2, \dots ,U_p$ in (\ref{cat1}) contribute to the number of iterations resulting from Lemma \ref{LemmaB}. Therefore, the $(-1)^{p-1}$ corresponds to the choices $\psi_1 ,\dots ,\psi_i$ for $\psi_1 +\cdots +\psi_i =\sigma_1 +\cdots +\sigma_{p-1}$ and also from the $W\subseteq f(\sigma')$ corresponding to the $U_{2} ,\dots ,U_{p-1}$. We can now rewrite (\ref{cat1}) as
 \begin{align}
 &-{\check{\gamma}_{\tau}}^{\wt [v]}\sum_{k=1}^N A_{k}^{\nu} \sum_{\sigma \in I(U_1)} \sum_{\sigma' +\sigma'' =\sigma} \sum_{\substack{i\geq 1 \\ \psi_1 +\cdots +\psi_i =\sigma'}} (-1)^i E_{\psi_1}\cdots E_{\psi_i} D_{\sigma''} \notag \\
 &\hspace{10mm} \times \sum_{U\subseteq W\subseteq f(\sigma)} \sum_{\substack{j\geq 1 \\ H_{U_{j+1}}^{\sigma_{j}} \cdots H_{U_2}^{\sigma_1} =H_W^\sigma}} (-1)^j S^k \left( \prod_{r\in U} {\check {\gamma }_{\tau }}o(v_r); \left\{F_{U}^{W}H_{W}^{\sigma}v,x\right\} ,\tau \right), \label{cat100}
 \end{align}
 where the minus sign in front corresponds to the final iteration which introduces the $D_{\sigma''}$ and $F_U^W$ terms (and also the $E_{\sigma''}$ and $H_U^W$ terms which are now contained in the expression corresponding to $\sigma'' =1$ and $U=W=\emptyset$). Therefore, by the discussion above and the fact $\sum_{j\geq 1}^N (-1)^j \binom{N}{j} =-(-1)^N$, we find (\ref{cat100}) becomes
  \begin{align*}
 &-{\check{\gamma}_{\tau}}^{\wt [v]}\sum_{k=1}^N A_{k}^{\nu} \sum_{\sigma \in I(U_1)} \sum_{\sigma' +\sigma'' =\sigma} (-1)^{\frac{\lvert m(\sigma')\rvert}{2}} E_{\sigma'} D_{\sigma''} \\
 &\hspace{15mm} \times \sum_{U\subseteq W\subseteq f(\sigma)} \sum_{j\geq 1} (-1)^j \binom{\lvert f(\sigma)\setminus W \rvert}{j} S^k \left( \prod_{r\in U} {\check {\gamma }_{\tau }}o(v_r); \left\{F_{U}^{W}H_{W}^{\sigma}v,x\right\} ,\tau \right) \\
 &\hspace{5mm}={\check{\gamma}_{\tau}}^{\wt [v]}\sum_{k=1}^N A_{k}^{\nu} \sum_{\sigma \in I(U_1)} \sum_{\sigma_1 +\sigma_2 =\sigma} (-1)^{\frac{\lvert m(\sigma_1)\rvert}{2}} E_{\sigma_1} D_{\sigma_2} \\
 &\hspace{15mm} \times \sum_{U\subseteq W\subseteq f(\sigma)} (-1)^{\lvert f(\sigma)\setminus W \rvert} S^k \left( \prod_{r\in U} {\check {\gamma }_{\tau }}o(v_r); \left\{F_{U}^{W}H_{W}^{\sigma}v,x\right\} ,\tau \right).
 \end{align*}
 Applying Lemma \ref{LemmaD} under the assumption $v_{s}(0)v=0$ gives the desired result.
 \end{proof}
 

\subsection{Step five\label{Step5}}

 Before proving Theorem \ref{ThmMain}, we first establish the following lemma.
 
\begin{lemma}\label{LemNew}
 Suppose $J(0)v =\alpha v$ and $K(0)v=\beta v$ for $\alpha ,\beta \in \mathbb{C}$. If $\alpha \not =0$ or $\beta \not =0$, then $\Phi_j \left(v:(J,K),\tau \right)=0$ for all $\tau \in \mathbb{H}$ and $1\leq j\leq N$.
\end{lemma}

\begin{proof}
 Since
 \[
 \tr_{M^j} e^{2\pi i o(J)} o(v) q^{o(\widetilde \omega)+o(K)} = q_{\alpha} \tr_{M^j} o(v) e^{2\pi i o(J)} q^{o(\widetilde \omega)+o(K)}
 =q_{\alpha} q^{\beta} \tr_{M^j} e^{2\pi i o(J)} o(v) q^{o(\widetilde \omega)+o(K)},
 \]
 we have
 \[
 \left(1-q_{\alpha} q^{\beta}\right) \tr_{M^j} e^{2\pi i o(J)} o(v) q^{o(\widetilde \omega)+o(K)} =0.
 \]
 Therefore, $\Phi_j \left(v:(J,K),\tau \right)=0$ for all $\tau \in \mathbb{H}$ such that $\alpha +\tau \beta \not \in \mathbb{Z}$. It follows that $\Phi_j \left(v:(J,K),\tau \right)=0$ on an open ball in $\mathbb{H}$ and by our assumption that $\Phi_j$ is convergent on $\mathbb{H}$, this extends to all of $\mathbb{H}$.
\end{proof}

 We are now in position to prove Theorem \ref{ThmMain} and Corollary \ref{ThmMain2}.\\
 
 \noindent \textit{Proof of Theorem \ref{ThmMain}.} \,\,\,
 We begin by mimicking the proof of the Main Theorem in \cite{Miy-Theta} and fix $s,t \geq 0$ so that $s+t=n$, and assume
 \[
 v_j =\begin{cases} J &\text{ if } 1\leq j\leq s \\ K &\text{ if } s+1 \leq j \leq t. \end{cases}
 \]
 For $\sigma \in I(n)$, partition the set $\Omega =\Omega_n =\{ 1,\dots ,n\}$ by setting
 \begin{align*}
 f_{1}(\sigma) := \{ j\in \Omega \mid j=\sigma (j) \leq s\}, \hspace{10mm}  m_{11}(\sigma) &:= \{ j\in \Omega \mid j<\sigma (j) \leq s \},   \\
 f_{2}(\sigma) := \{j\in \Omega \mid s<j=\sigma (j) \}, \hspace{10mm} m_{12}(\sigma) &:= \{ j\in \Omega \mid j \leq s < \sigma (j) \},  \\
 \text{ and } \hspace{10mm} m_{22}(\sigma) &:= \{ j\in \Omega \mid s< j <\sigma (j) \}.
 \end{align*}
  For $p,q,r\in \mathbb{Z}$, set $m_1 := s-2p-r \geq 0$ and $m_2 :=t-2q -r \geq 0$.\\
 \indent There are $\binom{s}{r}$ and $\binom{t}{r}$ many ways to choose $r$ elements from $s$ and $t$ many elements, respectively. Having chosen $r$ many elements from both, there are $\binom{2p +m_1}{m_1} \frac{(2p)!}{p! 2^p}$ many ways to choose $\sigma$ such that $\lvert m_{11}(\sigma) \rvert = 2p$ and $\lvert f_{1}(\sigma) \rvert =m_1$ on $\{1, \dots ,s\}$, excluding the $r$ many elements in $m_{12}(\sigma)$ since $s-r=2p+k$. Similarly, there are $\binom{2q+m_2}{m_2} \frac{(2q)!}{q! 2^q}$ many such $\sigma$ with $\lvert m_{22}(\sigma) \rvert =2q$ and $\lvert f_2 (\sigma) \rvert = m_2$. Finally, there are $r!$ many ways which a $\sigma \in I(n)$ under these restrictions can place the $r$ many elements  $j\in m_{12}(\sigma)$ with $j\leq s$ into the $r$ spots above $s$. Therefore, the total number of $\sigma \in I(n)$ with $\lvert m_{11}(\sigma)\rvert /2=p$, $\lvert m_{12}(\sigma)\rvert /2=r$, $\lvert m_{22}(\sigma)\rvert /2=q$, $\lvert f_{1}(\sigma)\rvert =m_1$, and $\lvert f_{2}(\sigma)\rvert =m_2$ is
 \[
 \binom{s}{r}\binom{2p +m_1}{m_1} \frac{(2p)!}{p! 2^p} \binom{t}{r}\binom{2q+m_2}{m_2} \frac{(2q)!}{q! 2^q} r! =\frac{(2p+r+m_1)!(2q+r+m_2)!}{m_1 !m_2 ! r!p!q! 2^{p+q}}.
 \]
 By Lemma \ref{LemNew}, we may assume $o(J)v=o(K)v=0$. Using Proposition \ref{HardThm1}, we note
 \begin{align*}
 &S^{\nu} \left( o(J)^s o(K)^t ; \left\{v,\frac{x}{c\tau +d} \right\}, \frac{a\tau +b}{c\tau +d} \right) \notag \\
 &= \sum_{k=1}^N A_{k}^{\nu} \sum_{\sigma \in I(s+t)} \left(\frac{\langle J,J\rangle c{\check{\gamma}_{\tau}}}{2\pi i}\right)^{\lvert m_{11}(\sigma)\rvert} \left(\frac{\langle J,K\rangle c{\check{\gamma}_{\tau}}}{2\pi i}\right)^{\lvert m_{12}(\sigma)\rvert} \left(\frac{\langle K,K\rangle c{\check{\gamma}_{\tau}}}{2\pi i}\right)^{\lvert m_{22}(\sigma)\rvert} \notag \\
 &\hspace{5mm} \times \sum_{i=0}^{m_1} \sum_{j=0}^{m_2} S^k \left( \left[{\check{\gamma}_{\tau}}o(J)\right]^i \left[{\check{\gamma}_{\tau}}o(K)\right]^j ;\left\{ \left[\frac{c}{2\pi i}J[1]\right]^{m_1 -i} \left[\frac{c}{2\pi i}K[1]\right]^{m_2 -j} v, x\right\},\tau \right) \notag \\
&= \sum_{k=1}^N A_{k}^{\nu} \sum_{p,q,r} \frac{(2p+r+m_1)!(2q+r+m_2)!}{m_1 !m_2 ! r!p!q! 2^{p+q}} \notag \\
&\hspace{5mm} \times \left(\frac{\langle J,J\rangle c{\check{\gamma}_{\tau}}}{2\pi i}\right)^{p} \left(\frac{\langle J,K\rangle c{\check{\gamma}_{\tau}}}{2\pi i}\right)^{r} \left(\frac{\langle K,K\rangle c{\check{\gamma}_{\tau}}}{2\pi i}\right)^{q} \notag \\
 &\hspace{5mm} \times \sum_{i=0}^{m_1} \sum_{j=0}^{m_2} S^k \left( \left[{\check{\gamma}_{\tau}}o(J)\right]^i \left[{\check{\gamma}_{\tau}}o(K)\right]^j ;\left\{ \left[\frac{c}{2\pi i}J[1]\right]^{m_1 -i} \left[\frac{c}{2\pi i}K[1]\right]^{m_2 -j} v, x\right\},\tau \right). \notag 
 \end{align*}
 Therefore,
 \begin{align*}
 &S^{\nu} \left( e^{2\pi i \left(o(J)+\frac{\langle J,K \rangle}{2} +\gamma \tau \left(o(K) +\frac{\langle K,K\rangle}{2}\right) \right)} ; \left\{v,\frac{x}{c\tau +d}\right\}, \frac{a\tau +b}{c\tau +d} \right) \notag \\
 &=\sum_{\ell_1 ,\ell_2 ,\ell_3 ,\ell_4} \frac{(2\pi i)^{\ell_1 +\ell_2 +\ell_3 +\ell_4}}{\ell_1 ! \ell_2 ! \ell_3 ! \ell_4 !}  \left(\frac{\langle J,K\rangle}{2}\right)^{\ell_3} \left(\gamma \tau \frac{\langle K,K\rangle}{2}\right)^{\ell_4} \notag \\
 &\hspace{5mm}  \times \sum_{k=1}^N A_{k}^{\nu} \sum_{\substack{p,q,r,m_1 ,m_2 \\ 2p+r+m_1 =\ell_1 \\2q+r+m_2 =\ell_2}} \frac{(2p+r+m_1)!(2q+r+m_2)!}{m_1 !m_2 ! r!p!q! 2^{p+q}} \left( \gamma \tau \right)^{2q+r+m_2} \notag \\
&\hspace{5mm} \times \left(\frac{\langle J,J\rangle c{\check{\gamma}_{\tau}}}{2\pi i}\right)^{p} \left(\frac{\langle J,K\rangle c{\check{\gamma}_{\tau}}}{2\pi i}\right)^{r} \left(\frac{\langle K,K\rangle c{\check{\gamma}_{\tau}}}{2\pi i}\right)^{q} \notag \\
 &\hspace{5mm} \times \sum_{i=0}^{m_1} \sum_{j=0}^{m_2} S^k \left( \left[{\check{\gamma}_{\tau}}o(J)\right]^i \left[{\check{\gamma}_{\tau}}o(K)\right]^j ;\left\{ \left[\frac{c}{2\pi i}J[1]\right]^{m_1 -i} \left[\frac{c}{2\pi i}K[1]\right]^{m_2 -j} v, x\right\},\tau \right) \notag \\
 &=\sum_{p,q,r,m_1 ,m_2,\ell_3 ,\ell_4} \frac{(2\pi i)^{p+q+r+m_1 +m_2+\ell_3 +\ell_4}}{m_1 !m_2 ! r!p!q! \ell_3 ! \ell_4 !}  \left(\frac{\langle J,K\rangle}{2}\right)^{\ell_3} \left(\gamma \tau \frac{\langle K,K\rangle}{2}\right)^{\ell_4} \notag \\
 &\hspace{5mm}  \times \sum_{k=1}^N A_{k}^{\nu} \left( \gamma \tau \right)^{2q+r+m_2} \left(\frac{\langle J,J\rangle c{\check{\gamma}_{\tau}}}{2}\right)^{p} \left(\langle J,K\rangle c{\check{\gamma}_{\tau}}\right)^{r} \left(\frac{\langle K,K\rangle c{\check{\gamma}_{\tau}}}{2}\right)^{q} \notag \\
 &\hspace{5mm} \times \sum_{i=0}^{m_1} \sum_{j=0}^{m_2} S^k \left( \left[{\check{\gamma}_{\tau}}o(J)\right]^i \left[{\check{\gamma}_{\tau}}o(K)\right]^j ;\left\{ \left[\frac{c}{2\pi i}J[1]\right]^{m_1 -i} \left[\frac{c}{2\pi i}K[1]\right]^{m_2 -j} v, x\right\},\tau \right) \notag \\
 &=\sum_{p,q,r,m_1 ,m_2,\ell_3 ,\ell_4} \frac{(2\pi i)^{p+q+r+m_1 +m_2+\ell_3 +\ell_4}}{m_1 !m_2 ! r!p!q! \ell_3 ! \ell_4 !}  \left(\frac{\langle J,K\rangle}{2}\right)^{\ell_3} \left(\gamma \tau \frac{\langle K,K\rangle}{2}\right)^{\ell_4} \notag \\
 &\hspace{5mm}  \times \sum_{k=1}^N A_{k}^{\nu} \left(\frac{\langle J,J\rangle c{\check{\gamma}_{\tau}}}{2}\right)^{p} \left(\langle J,K\rangle c{\hat{\gamma}^{\tau}}\right)^{r} \left(\frac{\langle K,K\rangle c\left({\hat{\gamma}^{\tau}}\right)^2}{2{\check{\gamma}_{\tau}}}\right)^{q} \notag \\
 &\hspace{5mm} \times \sum_{i=0}^{m_1} \sum_{j=0}^{m_2} S^k \left( \left[{\check{\gamma}_{\tau}}o(J)\right]^i \left[{\hat{\gamma}^{\tau}}o(K)\right]^j ;\left\{ \left[\frac{c}{2\pi i}J[1]\right]^{m_1 -i} \left[\frac{c{\hat{\gamma}^{\tau}}}{2\pi i {\check{\gamma}_{\tau}}}K[1]\right]^{m_2 -j} v, x\right\},\tau \right) \notag \\
 &=\sum_{k=1}^N A_{k}^{\nu} S^k \biggl ( e^{2\pi i \left[{\check{\gamma}_{\tau}}o(J) +{\hat{\gamma}^{\tau}}o(K)  \right]} e^{2\pi i \left[ \frac{\langle K,K\rangle}{2} \left( \gamma \tau +\frac{c\left({\hat{\gamma}^{\tau}}\right)^2}{\check{\gamma}_{\tau}}\right) + \frac{\langle J,K \rangle}{2} \left( 1+2c{\check{\gamma}_{\tau}}\right) +\frac{\langle J,J\rangle}{2} c{\check{\gamma}_{\tau}} \right] }   ; \notag \\
 &\hspace{10mm} \left\{ e^{cJ[1]} e^{\frac{c{\hat{\gamma}^{\tau}}}{c\tau +d} K[1]} v,x\right\} ,\tau \biggl ) \notag \\
 &=\sum_{k=1}^N A_{k}^{\nu} S^k \biggl ( e^{2\pi i \left(o(bK+dJ) +\frac{1}{2}\langle bK+dJ, aK+cJ\rangle \right)} q^{\left(o(aK +cJ) +\frac{1}{2}\langle aK +cJ,aK+cJ\rangle  \right)}  ; \left\{ e^{cJ[1]} e^{\frac{{a\tau +b}}{c\tau +d} cK[1]} v,x\right\} ,\tau \biggl ),
 \end{align*}
 where we used that
 \[
 \frac{a\tau +b}{c\tau +d} +\frac{c\left(a\tau +b \right)^2}{c\tau +d} = a{\hat{\gamma}^{\tau}},
 \,\,\,\,\, 1+2c(c\tau +d) =2ac\tau +cd +ad,
 \]
 and
 \begin{equation}
 \frac{c(a\tau +b)}{c\tau +d} =a-\frac{1}{c\tau +d} \label{baaah}
 \end{equation}
 for $\gamma \in \text{SL}_2(\mathbb{Z})$. It follows that
 \begin{align}
 \Phi_{\nu} &\left(v \colon \left(J,K\right), \frac{a\tau +b}{c\tau +d}\right) ={\check{\gamma}_{\tau}}^{\wt [v]} \sum_{k=1}^N A_{k}^{\nu} \Phi_k \left( e^{cJ[1]} e^{\frac{{a\tau +b}}{c\tau +d} cK[1]} v \colon \left(bK+dJ , aK+cJ\right) ,\tau \right) \notag \\
 &={\check{\gamma}_{\tau}}^{\wt [v]}\sum_{k=1}^N A_{k}^{\nu} \Phi_k \left( e^{(aK+cJ)[1]} e^{-\frac{1}{c\tau +d} K[1]} v \colon \left(bK+dJ , aK+cJ\right) ,\tau \right) \notag\\
 &=\sum_{n=0}^{\wt [v]} {\check{\gamma}_{\tau}}^{\wt [v] -n}\sum_{k=1}^N A_{k}^{\nu} \frac{1}{n!} \Phi_k \left( e^{(aK+cJ)[1]} \left(-K[1]\right)^n v \colon \left(bK+dJ , aK+cJ\right) ,\tau \right), \notag
 \end{align}
 proving Theorem \ref{ThmMain}. \hfill \qed \\
 
We complete this section by proving Corollary \ref{ThmMain2}.\\
 
 \noindent \textit{Proof of Corollary \ref{ThmMain2}.} \,\,\,
 Since $K[1]^\ell v \in V_{[\wt [v]-\ell]}$, equation (\ref{KTrans1}) gives
 \begin{align*}
 &\Phi_j \left(e^{K[1]} v:(J,K),\frac{a\tau +b}{c\tau +d} \right)= \sum_{\ell =0}^\infty \frac{1}{\ell !} \Phi_j \left(K[1]^\ell v:(J,K),\frac{a\tau +b}{c\tau +d} \right) \\
 &=\sum_{\ell =0}^\infty (c\tau +d)^{\wt [v] -\ell} \sum_{k=1}^N A_{j,k}^\gamma \Phi_k \left(e^{cJ[1]+\frac{a\tau +b}{c\tau +d}cK[1]} \frac{K[1]^\ell}{\ell !}v:(bK+dJ,aK+cJ),\tau \right) \\
 &=\sum_{\ell =0}^\infty (c\tau +d)^{\wt [v]} \sum_{k=1}^N A_{j,k}^\gamma \Phi_k \left(e^{aK[1]+ cJ[1]-\frac{1}{c\tau +d}K[1]} \frac{1}{\ell !}\left(\frac{K[1]}{c\tau +d}\right)^\ell v:(bK+dJ,aK+cJ),\tau \right) \\
 &= (c\tau +d)^{\wt [v]} \sum_{k=1}^N A_{j,k}^\gamma \Phi_k \left(e^{aK[1]+ cJ[1]-\frac{1}{c\tau +d}K[1] +\frac{1}{c\tau +d}K[1]} v:(bK+dJ,aK+cJ),\tau \right) \\
 &=(c\tau +d)^{\wt [v]} \sum_{k=1}^N A_{j,k}^\gamma \Phi_k \left(e^{aK[1]+cJ[1]} v:(J,K),\tau \right),
 \end{align*}
 as desired. \hfill \qed
 

\section{Application: lattice VOAs and theta functions\label{SectionExamples}}

 In this section we use the theory of VOAs, and in particular Theorem \ref{ThmMain}, to provide another proof of the modular transformation laws for derivatives of Jacobi theta functions. In doing so, we also study an example of a one-point theta function for an element of a VOA other than $v=\textbf{1}$.

 Let $V=V_{2\mathbb{Z}\alpha}$ be the lattice VOA constructed from the $1$-dimensional positive definite even lattice $2\mathbb{Z}\alpha$, where $\langle \alpha ,\alpha \rangle =1$. It is known (see \cite{Dong-Lattice}) that $V$ has the four inequivalent irreducible modules
 \[
 M^0=V, \hspace{5mm} M^1 =V_{\left(2\mathbb{Z}+\frac{1}{2}\right)\alpha}, \hspace{5mm} M^2 =V_{\left(2\mathbb{Z}+1\right)\alpha}, \hspace{5mm}\text{and}\hspace{5mm} M^3 =V_{\left(2\mathbb{Z}-\frac{1}{2}\right)\alpha}.
 \]
 Consider the Jacobi theta functions defined as 
 \[
 \vartheta_{hk} (\tau ,z):=\sum_{n\in \mathbb{Z}} e^{\pi i \left(n+\frac{h}{2}\right)^2 \tau +2\pi i\left(n+\frac{h}{2}\right)\left(z+\frac{k}{2}\right)},
 \]
 where $h,k=0,1$ (see \cite{Mumford}, for example, for more details). The transformation laws with respect to the matrix $S=\left(\begin{smallmatrix} 0&-1\\1&0 \end{smallmatrix}\right)$ are
 \[
 \vartheta_{hk} \left(-\frac{1}{\tau}, \frac{z}{\tau} \right) = i^{hk} (-i\tau)^{\frac{1}{2}} e^{\pi i \frac{z^2}{\tau}} \vartheta_{kh} (\tau ,z).
 \]
 Recall that $\alpha [n] \alpha  =\delta_{n,1}\alpha$ for $n \geq 0$. Set $D_x :=\frac{1}{2\pi i} \frac{d}{dx}$ for a variable $x$ and let $\vartheta' (\tau ,z):=D_z \vartheta (\tau ,z)$. Note that $D_z \Phi_{j} (\textbf{1}: \{z\alpha ,0 \},\tau )= \Phi_{j} (\alpha: \{z\alpha ,0 \},\tau )$ for any $j$. It can be shown that
 \[
 \vartheta_{hk} (\tau ,z)=\eta (\tau)i^{hk} \left(\Phi_{h} (\textbf{1}: \{z\alpha ,0 \},\tau)+(-1)^{k} \Phi_{2+h}(\textbf{1}: \{z\alpha ,0 \},\tau)\right),
 \]
 where $\eta (\tau)$ is the Dedekind eta-function and transforms as $\eta \left(-\frac{1}{\tau}\right) = (-i\tau)^{\frac{1}{2}} \eta (\tau)$. It also follows that
 \[
 \vartheta_{hk}' (\tau ,z)=\eta (\tau)i^{hk} \left(\Phi_{h} (\alpha : \{z\alpha ,0 \},\tau)+(-1)^{k} \Phi_{2+h}(\alpha : \{z\alpha ,0 \},\tau)\right).
 \]
 Meanwhile, the values $A_{h,j}^\gamma$, $0\leq j \leq 3$, from Theorem \ref{ThmMain} (where we have began our indexing of modules with $0$ here) are known when $\gamma$ is $S$ or the matrix $T=\left(\begin{smallmatrix} 1&1\\ 0&1 \end{smallmatrix}\right)$ (see \cite[Example $1$]{Gannon-Data}, for example). The values for the $S$-matrix are $(S_h^j):=(A_{h,j}^S )=(\frac{1}{2}e^{\frac{\pi ihj}{2}})$. Using Theorem \ref{ThmMain}, we find
 \begin{align*}
 \vartheta_{hk}' &\left(-\frac{1}{\tau}, \frac{z}{\tau}\right)=i^{hk}\eta \left(-\frac{1}{\tau}\right)\left[ \Phi_{h} \left(\alpha : \left(\frac{z\alpha}{\tau} ,0 \right),-\frac{1}{\tau} \right)+(-1)^{k} \Phi_{2+h}\left(\alpha : \left(\frac{z\alpha}{\tau} ,0 \right),-\frac{1}{\tau} \right) \right]\\
&=i^{hk}(-i\tau)^{\frac{1}{2}} \tau \eta (\tau)\sum_{j=0}^3 \left[ \left(S_{h}^{j}+(-1)^k S_{2+h}^{j}\right)\Phi_{j} \left(\alpha + \frac{z}{\tau}\textbf{1} : \left(0,\frac{z\alpha}{\tau} \right),\tau \right) \right]\\
&= (-i\tau)^{\frac{1}{2}}e^{\pi i \frac{z^2}{\tau}}  \tau \eta (\tau) i^{hk}  \left(\Phi_k \left(\alpha + \frac{z}{\tau}\textbf{1} : \left(z\alpha ,0 \right),\tau \right) +(-1)^h \Phi_{2+k}  \left(\alpha + \frac{z}{\tau}\textbf{1} : \left(z\alpha ,0 \right),\tau \right) \right)  \\
&= \tau (-i\tau)^{\frac{1}{2}}e^{\pi i \frac{z^2}{\tau}}  \left( \vartheta_{kh}'(\tau ,z) + \vartheta_{kh}(\tau ,z)\left(\frac{z}{\tau}\right) \right) .
 \end{align*}
 Additionally, using the values for the $T$-matrix $(T_h^j):=(A_{h,j}^T ) =(\delta_{h,j} e^{\frac{\pi i hj}{4} -\frac{\pi i}{12}})$, where $\delta_{h,j}$ is $1$ if $h=j$ and $0$ otherwise, we find
 \begin{align*}
 \vartheta_{hk}' &(\tau +1,z) =i^{hk} \eta (\tau +1) \left[ \Phi_h (\alpha :(z\alpha ,0),\tau +1) +(-1)^k \Phi_{2+h} (\alpha : (z\alpha ,0),\tau +1) \right] \\
 &=i^{hk} e^{\frac{\pi i}{12}}\eta (\tau ) \sum_{j=0}^3 \left[ T^j_h \Phi_j (\alpha :(z\alpha ,0),\tau) +(-1)^k T^j_{2+h} \Phi_{j} (\alpha : (z\alpha ,0),\tau) \right] \\
 &=i^{hk} \eta (\tau ) \left(\frac{\sqrt{2}}{2} +i\frac{\sqrt{2}}{2}\right)^h \left( \Phi_h (\alpha :(z\alpha ,0),\tau) - i^h (-1)^k \Phi_{2+h} (\alpha :(z\alpha ,0),\tau) \right) \\
  &= \delta_{h,0} \left(\delta_{k,0} \vartheta_{01}' (\tau ,z) + \delta_{k,1} \vartheta_{00}' (\tau ,z)\right) +\delta_{h,1} \frac{\sqrt{2}}{2} \left( \vartheta_{11}'(\tau ,z) +(-1)^k \vartheta_{10}'(\tau ,z) \right),
 \end{align*}
 which establishes the modular transformation laws for derivatives of Jacobi theta functions.

\end{document}